\documentclass[12pt,a4paper,reqno]{amsart}
\usepackage{amsfonts,amssymb,times,mathptmx,hyperref}

\providecommand{\BBb}[1]{{\mathbb{#1}}}
\providecommand{\cal}[1]{{\mathcal{#1}}} 

\newcommand{\B}{{\BBb B}}
\newcommand{\C}{{\BBb C}}
\newcommand{\dual}[2]{\langle\,#1,\,#2\,\rangle}
\newcommand{\fracc}[2]{{
                \textstyle\frac{#1}{\raise 1pt\hbox{$\scriptstyle #2$}}}}
\newcommand{\fracp}{\fracc1p}
\newcommand{\fracnp}{\fracc np}
\newcommand{\fracci}[2]{{\frac{#1}{\raise 1pt\hbox{$\scriptscriptstyle #2$}}}}
\newcommand{\fracpi}{\fracci1p}
\newcommand{\im}{\operatorname{i}}
\newcommand{\loc}{\operatorname{loc}}
\newcommand{\nrm}[2]{\|#1\|_{#2}}
\newcommand{\Nrm}[2]{\bigl\|#1\bigr\|_{#2}}
\newcommand{\norm}[2]{\mathinner{\|}#1\,|#2\|}
\newcommand{\Norm}[2]{\mathinner{\bigl\|\,#1\,\big|#2\bigr\|}}
\newcommand{\op}[1]{\operatorname{#1}}
\newcommand{\N}{\BBb N}

\newcommand{\R}{{\BBb R}}
\newcommand{\Rn}{{\BBb R}^{n}}

\newcommand{\scal}[2]{(\,#1\,|\, #2\,)}

\newcommand{\supp}{\operatorname{supp}}
\newcommand{\Z}{\BBb Z}

\renewcommand{\hat}[1]{\overset{{\scriptscriptstyle \wedge}}{#1}}

\numberwithin{equation}{section}
\newtheorem{thm}{Theorem}
\numberwithin{thm}{section}
\newtheorem{prop}[thm]{Proposition}
\newtheorem{lem}[thm]{Lemma}
\newtheorem{cor}[thm]{Corollary}

\newtheorem*{tthm}{{\bf Theorem\/}}

\theoremstyle{definition}

\theoremstyle{remark}
\newtheorem{rem}[thm]{Remark}
\title[PS.D.O.s and Triebel--Lizorkin spaces]{Domains 
of pseudo-differential operators:\\ a case for the
Triebel--Lizorkin spaces}

\author{Jon Johnsen}
\address{Dept. of Mathematics\\
Aalborg University\\
Fredrik Bajers Vej~7G\\ DK-9220 Aalborg {\O}st\\ Denmark}
\email{jjohnsen@math.aau.dk}

\begin{document}

 \begin{abstract}
The main result is that every pseudo-differential operator
of type $1,1$ and order $d$ is continuous from the Triebel--Lizorkin space
$F^d_{p,1}$ to $L_p$, $1\le p<\infty$, and that this is optimal within
the Besov and Triebel--Lizorkin scales. The proof also leads to the
known continuity for $s>d$, while for all real $s$
the sufficiency of H{\"o}rmander's condition on the twisted diagonal is
carried over to the Besov and Triebel--Lizorkin framework.
To obtain this, type $1,1$-operators are extended to distributions with
compact spectrum, 
and Fourier transformed operators of this type are on such distributions
proved to satisfy a support rule, 
generalising the rule for convolutions.
Thereby the use of reduced symbols, as introduced by Coifman and Meyer,
is replaced by direct application of the paradifferential methods.
A few flaws in the literature have been detected and corrected.
 \end{abstract}

\maketitle
\section{Introduction}   \label{intr-sect}

At first glance this article's title may seem rather unmotivated:
for symbols $a$ in H{\"o}rmander's class
$S^{d}_{\rho,\delta}(\Rn\times\Rn)$, ie for $a\in C^\infty(\R^{2n})$ such
that 
\begin{equation}
  |D^\alpha_\xi D^\beta_x a(x,\xi)|\le 
   C_{\alpha\beta} (1+|\xi|)^{d-\rho|\alpha|+\delta|\beta|},
  \label{Sdrd-eq}
\end{equation}
it is well known that for $0\le \delta\le\rho\le1$
the operators 
\begin{equation}
  a(x,D)u(x)=\op{OP}(a)u(x)=(2\pi)^{-n}\!\int 
  e^{\im x\cdot\xi} a(x,\xi)\hat u(\xi)\,d\xi
  \label{axD-eq}
\end{equation}
map the Schwartz space $\cal S(\Rn)$ continuously into itself. 
For $(\rho,\delta)\ne (1,1)$ the operators form a class 
invariant under passage to adjoints, 
and they extend in this way to continuous, `globally' defined operators 
\begin{equation}
  a(x,D)\colon \cal S'(\Rn)\to\cal S'(\Rn).
  \label{aS-eq}
\end{equation}

But for $\rho=\delta=1$ the domain situation is different,
for Ching \cite{Chi72} showed the
existence of $a\in S^0_{1,1}$ such that $a(x,D)$ doesn't belong to
$\B(L_2(\Rn))$. That all operators in $\op{OP}(S^0_{1,1})$ are bounded
on $C^s$ and $H^s$ for $s>0$ was first proved by Stein, albeit in
unpublished work 
(cf Meyer \cite{Mey80} resp.\ H{\"o}rmander \cite{H88} for this). 
Continuity $H^{s+d}_p \to H^{s}_p$ for $s>0$, $1<p<\infty$ is due to Meyer
\cite{Mey80,Mey81}. 

Bourdaud analysed adjoints of $\op{OP}(S^0_{1,1})$, and
\cite[Thm.~3]{Bou88} lead to criteria 
for a given $S^0_{1,1}$-operator to 
be bounded on $H^s_p$ for \emph{all} $s\in\R$.
For $d\in\R$ and $p=2$,
H{\"o}rmander related this question more directly to the symbol's
properties, eg via the following sufficient condition: if the
partially Fourier transformed symbol $\hat a(\xi,\eta)=\cal
F_{x\to\xi}a(x,\eta)$ 
vanishes in a conical neighbourhood of a non-compact part of
the \emph{twisted diagonal}
$\{\,(\xi,\eta)\mid \eta=-\xi\,\}$, ie for some constant
$C\ge1$ fulfils 
\begin{equation}
  \hat a(\xi,\eta)=0 \quad\text{for}\quad 
  C(|\xi+\eta|+1)\le |\eta|,
  \label{H-cnd}
\end{equation}
then $a(x,D)$ is bounded $H^{s+d}\to H^{s}$ for all $s\in\R$; cf \cite{H88}.

However, not all symbols $a\in S^d_{1,1}$ 
fulfill \eqref{H-cnd} (cf \cite{Chi72} or \eqref{ching-eq} below),
so it is natural to ask whether a maximal 
domain of definition of $a(x,D)$ exists; clearly there is no such among
the $H^s$ with $s>0$. The next result 
gives affirmative answers by means of the Triebel--Lizorkin scale
$F^{s}_{p,q}(\Rn)$. 

\begin{tthm}
  \label{Fp1-thm}
Every $a\in S^{d}_{1,1}(\Rn\times\Rn)$, $d\in\R$, 
yields a bounded operator 
\begin{align}
  a(x,D)&\colon F^{d}_{p,1}(\Rn)\to L_p(\Rn)
         \quad\text{for}\quad p\in[1,\infty[,
  \label{Fp1-eq}\\
  a(x,D)&\colon B^{d}_{\infty,1}(\Rn)\to L_\infty(\Rn).
  \label{Bi-eq}
\end{align}
The class $\op{OP}(S^d_{1,1})$ contains operators 
$a(x,D)\colon \cal S(\Rn)\to\cal
D'(\Rn)$, that are discontinuous when $\cal S(\Rn)$ is given the induced
topology from any of the Triebel--Lizorkin spaces $F^d_{p,q}(\Rn)$
or Besov spaces $B^d_{p,q}(\Rn)$ with $p\in [1,\infty]$ and
$q\in\,]1,\infty]$ (while $\cal D'$ has the usual topology).
\end{tthm}

In particular, for fixed $p\in[1,\infty[\,$, all operators in
$\op{OP}(S^d_{1,1})$ are
bounded $F^d_{p,1}\to L_p$, but on 
any larger space in the $B^{s}_{p,q}$- and $F^{s}_{p,q}$-scales they will
(whatever the codomain) in general only
be densely defined, unbounded.

To elucidate this, note that by the results cited above there is
continuity $H^{s+d}_p\to L_p$ for every $s>0$, but not in general for
$s=0$. It is well known that $H^s_p=F^s_{p,2}$ for $1<p<\infty$, $s\in \R$, 
so it could be natural to search for maximal domains among the more general 
Triebel--Lizorkin spaces $F^d_{p,q}$; here
$F^d_{p,1}$ is a candidate by \eqref{Fp1-eq}. On the larger spaces 
$F^d_{p,q}$ with $q>1$ the theorem yields
that operators in $\op{OP}(S^{d}_{1,1})$ cannot be
continuous. Moreover, in the Besov scale, $B^d_{p,1}\subset F^d_{p,1}$ for
$1\le p<\infty$, and also here spaces with $q>1$ are too large, in view of
the theorem. In this sense the theorem is sharp for $1\le
p<\infty$. 

\begin{rem}
  \label{Fspq-rem}
In $L_p$-theory of, say partial differential equations 
$H^s_p$-spaces are natural (eg $H^s_p=W^s_p$ for integer $s\ge0$),
but it is well known that other $L_p$-based scales must show up
too. Eg the trace $f(x',x_n)\mapsto f(x',0)$
is a surjection
\begin{equation}
  H^s_p(\Rn)\to B^{s-\fracpi}_{p,p}(\R^{n-1})
  \quad\text{for}\quad s>\fracp,\ 1<p<\infty;
\end{equation}
hence Besov spaces are \emph{inevitable}
in $L_p$-theory of boundary problems. 

Arguments in favour of \emph{Triebel--Lizorkin} spaces have,
perhaps, been less compelling. Although $F^s_{p,2}=H^s_p$ for
$1<p<\infty$, it could be argued that this need not make
the $F^{s}_{p,q}$-scale a \emph{useful} extension of 
the $H^s_p$-spaces; indeed, many properties of $F^{s}_{p,q}$ do not depend
on $q$, and some technicalities would be avoided by fixing $q=2$.
But the theorem shows that also $F^{s}_{p,q}$-spaces with $q=1$ 
are indispensable for a natural $L_p$-theory, also for $p=2$. 
\end{rem}

\subsection{Other mapping properties}
For continuity $F^{s+d}_{p,q}\to F^{s}_{p,q}$ with $s>0$, 
a few minor modifications of the
inequalities in the theorem's proof yield estimates implying
\eqref{Fs-eq}--\eqref{Bs-eq} below.
This proof should also be interesting because 
H{\"o}rmander's condition \eqref{H-cnd} is extended to the $F^{s}_{p,q}$- 
and $B^{s}_{p,q}$-scales by a mere \emph{addendum}
to the argument for \eqref{Fs-eq}--\eqref{Bs-eq}:

\begin{cor}
  \label{cont-cor}
Every $a(x,D)\in \op{OP}(S^{d}_{1,1}(\Rn\times\Rn))$ restricts for $s>0$
and $p$, $q\in[1,\infty]$ to a continuous map
\begin{align}
  a(x,D)&\colon F^{s+d}_{p,q}(\Rn)\to F^{s}_{p,q}(\Rn), 
    \quad\text{for}\quad p<\infty,
  \label{Fs-eq} \\
  a(x,D)&\colon B^{s+d}_{p,q}(\Rn)\to B^{s}_{p,q}(\Rn).
  \label{Bs-eq} 
\end{align}
If in addition \eqref{H-cnd} holds, then both \eqref{Fs-eq} and
\eqref{Bs-eq} are valid for all $s\in\R$.
\end{cor}
The corollary has a version with $p$, $q\in\,]0,\infty]$
if only $s>\max(0,\fracnp-n)$, as accounted for in
Section~\ref{quasi-sect} below (this partially removes a well-known obstacle
in the use of $F^{s}_{p,q}$-spaces).
A more far-reaching extension result is

\begin{prop}   \label{E'-prop}
Any $A$ in $\op{OP}(S^\infty_{1,1})$ is a map 
$A\colon \cal F^{-1}\cal E'(\Rn)\to \cal S'(\Rn)$, with range contained in 
$O_M(\Rn)$, ie in the space of $f\in C^\infty(\Rn)$ fulfilling estimates
$|D^\alpha f(x)|\le c_\alpha(1+|x|)^{N_\alpha}$ for all $\alpha\in N^n_0$.
\end{prop}

This shows that every type $1,1$-operator is
defined on a `large' space, and that the non-extendability to $\cal S'(\Rn)$
comes from distributions with ``high-frequency oscillations'' 
(corresponding to the fact that it is the distant part of the twisted
diagonal that matters).

\subsection{The methods of proof}
In Sections~\ref{brdl-sect}--\ref{cont-sect} below
the paradifferential approach is used for the proofs of the theorem and its
corollary. 
On the one hand, this strategy is well known and has been widely adopted for
$L_p$-questions, eg in works of Meyer, Bui Huy Qui,
Bourdaud,  Marschall and Yamazaki \cite{Mey80,Mey81,Bui,Bou88,Mar85,Y1}
(the list is by no means exhaustive), 
and here it was combined with the
density of \emph{reduced} symbols. This notion is due to Coifman and
Meyer \cite[Sec.~2.6]{CoMe78}, who in the proof of
\cite[Thm.~2.6.9]{CoMe78} used it to facilitate 
spectral estimation of terms like
$b(x,D)v$; in fact, reduced symbols have the form
$b(x,\xi)=\sum_{j=0}^\infty m_j(x)\varphi(2^{-j}\xi)$ for a 
$C^\infty$-function $\varphi$ supported in a corona around the origin 
and a bounded set of uniformly continuous
$L_\infty$-functions $m_j$, and for such symbols, inclusions of the support of
$\cal F(b(x,D)v)$ into balls and annuli was easily obtained.
 
On the other hand, however,
the combination of reduced symbols and paradifferential
techniques amounts to \emph{two} limit processes, 
which together make the action of $a(x,D)$ rather
intransparent. In order to avoid this drawback,
the arguments are here carried out \emph{directly} on the given symbols in 
$S^\infty_{1,1}$ and distributions $u$, without recourse to density of
reduced symbols or of Schwartz functions
(preferable since $\cal S(\Rn)$ is
not dense in eg $B^d_{\infty,1}(\Rn)$).
In doing so, the spectral estimates necessary for the paradifferential
approach are now obtained 
by means of Proposition~\ref{supp-prop} below.

Among the earlier contributions, reduced symbols are also not used in
\cite{Run85ex,Mar96},
but various flaws in these papers have been detected and corrected with the
present work; cf Remarks~\ref{Marsch-rem} and \ref{Runst-rem} below. 
 
To explain the direct approach in more detail, 
it is noted that
Corollary~\ref{cont-cor} also relies on convergence
criteria for series of distributions with spectral conditions, cf
Lemma~\ref{F-lem} below. It is therefore essential to have
control over the spectrum of $b(x,D)v$ for rather general $b$ and $v$.
For $b\in\cal S(\R^{2n})$ and $v\in\cal S(\Rn)$ this can be obtained at once,
since Fubini's theorem implies the well-known formula,
\begin{equation}
  \cal F(b(x,D)v)(\xi)
   =(2\pi)^{-n}\int \hat b(\xi-\eta,\eta)\hat v(\eta)\,d\eta.
  \label{Fbv-eq}
\end{equation}
For similar purposes H{\"o}rmander \cite[p.~1091]{H88}
extended \eqref{Fbv-eq} to symbols $b\in \cal S'(\R^{2n})$ with $v$
remaining in $\cal S(\Rn)$, noting that Schwartz' kernel theorem allows this
(one can eg apply \eqref{Fbv-eq} to a Schwartz function first).
But for the present direct treatment of symbols and distributions
\emph{both} in $\cal S'\setminus\cal S$, this approach does not suffice.
It is also difficult to use limiting procedures,
because $\cal S$ is not dense in $B^d_{\infty,1}$, eg since
$v\equiv1$ lies there; here $\hat v=(2\pi)^n\delta_0$ 
that moreover would be demanding to make sense of in \eqref{Fbv-eq}
when also $\hat b$ can be a singular distribution.

However, generalising a familiar convolution technique, one has
the following result that, despite its classical nature, could be important
for the future Littlewood--Paley analysis of pseudo-differential operators: 

\begin{prop}[the support rule]
  \label{supp-prop}
For $b\in S^\infty_{1,1}(\Rn\times\Rn)$,
\begin{equation}
  \supp \cal F(b(x,D)v) \subset
  \bigl\{\,\xi+\eta \bigm| (\xi,\eta)\in \supp \hat b(\cdot,\cdot),
            \eta\in\supp \hat v \,\bigr\},
  \label{supp-eq}
\end{equation}
for every $v\in\cal F^{-1}\cal E'(\Rn)$.
\end{prop}
Note that $v\in \cal F^{-1}\cal E'(\Rn)$ is meaningful 
for $b\in S^\infty_{1,1}$, by Proposition~\ref{E'-prop},
but that such $v$'s require more than \eqref{Fbv-eq}, since 
Proposition~\ref{E'-prop} contains no continuity, so that
eg density arguments are difficult to use.
(It is also not clear that \eqref{supp-eq} should follow
from results about wavefront sets, for the latter only account for
singularities in singular supports.)
Cf Section~\ref{supp-ssect} below for a proof that combines a convolution in
$\cal D'*\cal E'$ on $\R^{2n}$ with a trace argument. 

Somewhat surprisingly, the support rule seems to be hitherto undescribed in
the literature, even for classical symbols. (However, for reduced symbols
\eqref{supp-eq} is easy to obtain, as $\cal F(b(x,D)v)$ is a finite sum of
convolutions, $(2\pi)^{-n}\sum \hat m_j*(\varphi(2^{-j}\cdot) \hat v)$.)
At least for $b\in S^\infty_{1,1}$ the proposition  
is a novelty.

It is perhaps noteworthy that partially Fourier transformed symbols, such as
$\cal F_{x\to\xi} b(x,\eta)$, enter both the support rule \eqref{supp-eq}
and the twisted diagonal condition \eqref{H-cnd}. This could be natural
since \eqref{supp-eq} quite generally implies that the spectrum of $b(x,D)v$
cannot be larger than the combined frequencies in the symbol's $x$- and
$\eta$-dependencies. 

More specifically, Proposition~\ref{supp-prop} has as a
corollary, that if $b\in S^\infty_{1,0}$ and
$\supp\hat b\subset K'\times K''$ with $K''\Subset\Rn$, 
then
\begin{equation}
  \supp\cal F(b(x,D)v)\subset K'+K'',\quad\text{for}\quad  v\in\cal S'(\Rn).
\end{equation}
Indeed, for 
$\chi\in C^\infty_0$ such that $\chi\equiv1$ on $K''$, 
$b(x,D)v=b(x,D)\cal F^{-1}(\chi\hat v)$
by \eqref{axD-eq}--\eqref{aS-eq}, and $K'+K''$ results from \eqref{supp-eq}.

A brief review of the present paper has been given in
\cite{JJ04Dcr}.

\begin{rem}
  \label{hist-rem}
Consideration of $\op{OP}(S^d_{1,1}(\Rn\times\Rn))$ in
$F^{s}_{p,q}$-spaces was initiated by Runst \cite{Run85ex}, 
but unfortunately his proofs 
contained a flaw that one can correct by means of
Proposition~\ref{supp-prop}, cf Remark~\ref{Runst-rem} below. 
Seemingly Torres \cite{Tor90} was the first
to extend the $H^s_p$-continuity of \cite{Mey80,Mey81} to the
$F^s_{p,q}$-scale, using Frazier and Jawerth's $\varphi$-transformation 
\cite{FJ2}; Torres' results are improved in two respects cf
Remark~\ref{Torr-rem} below. 
The case $d=s=0$ was addressed by Bourdaud
\cite[Thm.~1]{Bou88},  
who showed continuity 
$B^0_{p,1}\to L_p$ for $1\le p\le\infty$; this is a special case of the
theorem since $B^0_{p,1}\subset F^0_{p,1}$ for $p\ge1$.
\end{rem}

\section{Linearisation and operators of type $1,1$}
  \label{Lin-sect}

The interest in type $1,1$-operators stems partly from the fact that they
appear in linearisations of non-linear functions. While settling 
the notation, this is recalled in the present section, and it is shown
that the theorem is easy to prove for operators in such linearisations.

\bigskip

When $F\in C^\infty(\R,\R)$ fulfils $D^kF\in L_\infty(\R)$ for $k\ge1$,
$F(0)=0$, the operator $u\mapsto F\circ u$, defined for
$u\in L_\infty(\Rn,\R)$, may be written as
\begin{equation}
  F(u(x))= a_u(x,D)u(x)
  \label{Fux-eq}
\end{equation}
for some $u$-dependent $a_u \in S^{0}_{1,1}(\Rn\times\Rn)$. To obtain this,
one may take a Littlewood--Paley decomposition $(\Phi_j)_{j\in \N_0}$, 
that is $\Phi_j\in  C^\infty(\Rn)$ with $\supp\Phi_j\Subset\Rn$
(where $A \Subset B$ means that $A$ has compact closure in $B$) and
\begin{gather}
  1\equiv \sum_{j=0}^\infty \Phi_j(\xi)  \\
 j>0\colon \xi\in\supp\Phi_j\implies \tfrac{11}{20}2^j\le |\xi|
   \le\tfrac{13}{10}2^j.
\end{gather}
Here one can set $\Phi_j(\xi)=\Psi_j(\xi)-\Psi_{j-1}(\xi)$ for $j\ge0$,
when $\Psi\in C^\infty(\R,\R)$ is chosen such that $\Psi(t)=1$ for
$t\le\tfrac{11}{10}$ and $\Psi(t)=0$ for $t\ge \tfrac{13}{10}$ and 
\begin{equation}
  \Psi_j(\xi)=\Psi(2^{-j}|\xi|), \qquad \Psi_{-1}\equiv0,
\end{equation}
for pointwisely $\Phi_0+\dots+\Phi_j=\Psi_j\to 1$. 
It is occasionally convenient to define eg
$\Phi=\Phi_j(2^j\cdot)$, which is independent of $j>0$; this extends to other
situations for simplicity's sake.

Using this, and setting
$v_j=\Phi_j(D)v$ and $v^j=\Psi_j(D)v$, one has $v=\sum_{j=0}^\infty v_j$ for
every tempered distribution $v$, and
\begin{equation}
  F(u(x))=\sum_{j=0}^\infty m_j(x)u_j(x)
  \label{Fum-eq}
\end{equation}
with multipliers $m_j(x)=\int_0^1 F'(u^{j-1}(x)+tu_j(x))\,dt$ as in
\cite{Mey81}. It was used there that \eqref{Fum-eq} for $u\in H^s_p$
with $s>\fracc np$ shows \eqref{Fux-eq} for
\begin{equation}
  a_u(x,D)v(x)=\sum_{j=0}^\infty m_j(x)\Phi_j(D)v(x).
  \label{au-id}
\end{equation}
Clearly $a_u(x,\xi)= \sum_{j=0}^\infty m_j(x)\Phi_j(\xi)$
is its symbol; the sum is locally finite, hence $C^\infty$.
Here $a_u\in S^{0}_{1,1}$, for if $j>0$ one has on $\supp \Phi_j$ 
that $2^j\doteq 1+|\xi|$ 
[ie, for some $c\ge1$, it holds that
$\tfrac{2^j}{c}\le|\xi|\le c2^j$],
so that $D^\alpha_\xi$ on $\Phi_j(\xi)=\Phi(2^{-j}\xi)$
produces the factor $2^{-j|\alpha|}$, estimated by
$c(1+|\xi|)^{-|\alpha|}$. Note also that
$\nrm{m_j}{\infty}\le\nrm{F'}{\infty}$ and $\nrm{D^\beta_xm_j}{\infty}
\le c(1+|\xi|)^{|\beta|}$.

Recall that the Triebel--Lizorkin space $F^{s}_{p,q}(\Rn)$ is defined
for $s\in\R$, $0<p<\infty$ and $0<q\le\infty$, as the set of 
$v\in\cal S'(\Rn)$ for which  
\begin{equation}
  \norm{v}{F^s_{p,q}}:=
  \nrm{(\sum_{j=0}^\infty 2^{sjq}|v_j(\cdot)|^q)^\fracci1q}{p}<\infty.
\end{equation}
This is a quasinorm for $p<1$ or $q<1$ (`quasi' will be suppressed below).
Here $\nrm{\cdot}{p}$ is the norm of $L_p(\Rn)$, $0<p\le\infty$, and for
$q=\infty$ the $\ell_q$-norm above should be replaced by an
$\ell_\infty$-norm (this is to be understood throughout when $q=\infty$
is included).  
The Besov space $B^{s}_{p,q}$ is defined by taking the $L_p$
norm of $2^{sj}v_j$ first, before the $\ell_q$-norm. 
General properties of the spaces are described in \cite{RuSi96, T2} or
\cite{Y1}. Below it is used that $F^t_{p,q}\hookrightarrow F^s_{p,q}$ for
$t>s$ and $F^s_{p,q}\hookrightarrow F^s_{p,r}$ for $0<q\le r\le\infty$.

For later reference, some important convergence criteria are
recalled from eg \cite[Prop.~2.3.2/2]{RuSi96} or \cite[Thm.~3.6--3.7]{Y1}, 
though for simplicity for $F^{s}_{p,q}$ and 
$p$, $q\in [1,\infty]$ only (cf Lemma~\ref{Fspq-lem} below).

\begin{lem}
  \label{F-lem}
Let $s>0$, $1\le p<\infty$ and $1\le q\le\infty$ and suppose 
$\sum_{j=0}^\infty u_j$ is a series 
in $\cal S'(\Rn)$ to which there exists $A>0$ such that, for $k\in \N_0$,
\begin{equation}
  \supp \cal F u_k\subset B(0,A2^k),\qquad
 F:=\Norm{(\sum_{j=0}^\infty
  2^{jsq}|u_j|^q)^{\fracci1q}(\cdot)}{L_p}<\infty.  
\end{equation}
Then $\sum u_j$ converges in $\cal S'(\Rn)$ to 
a limit $u\in F^{s}_{p,q}(\Rn)$  
and $\norm{u}{F^s_{p,q}}\le c\cdot F$ for a suitable constant $c$. 
Moreover, if $\cal F u_k(\xi)\ne0$ implies $\tfrac{1}{A}2^k\le|\xi|\le
A2^k$ for all $k>0$, then the result is valid for all $s\in\R$.
\end{lem}

\bigskip

The fact that $a_u(x,\xi)$ from \eqref{au-id} has the structure of a
reduced symbol immediately yields a
simple version of the theorem:

\begin{prop}
  \label{au-prop}
For $u\in L_\infty(\Rn,\R)$ with  $a_u\in
S^{0}_{1,1}(\Rn\times\Rn)$ defined as above, the associated operator 
is bounded, for every $p\in[1,\infty[\,$,
\begin{equation}
  a_u(x,D)\colon F^0_{p,1}(\Rn)\to L_p(\Rn),
\end{equation}
with $L_p$-convergence of \eqref{au-id} for every $v\in F^0_{p,1}(\Rn)$.
\end{prop}

\begin{proof}
For $v\in F^0_{p,1}$, one finds for any finite sum 
\begin{equation}
    \nrm{\sum m_j\Phi_j(D)v}{p}\le
    \nrm{\sum |m_j|\cdot|v_j|}{p}
    \le \nrm{F'}{\infty} \Nrm{\sum |v_j|}{p},
  \label{au-ineq}
\end{equation}
so it follows, with $\sum_{j=0}^\infty |v_j(x)|$ as a majorant, that the
series for $a_u(x,D)v(x)$ is fundamental in $L_p(\Rn)$, hence convergent. 
$a_u(x,D)v$ being defined thus, the above estimate may be read verbatim for
the sum over all $j\ge0$, which yields the claimed boundedness.
\end{proof}

The series defining $a_u(x,D)$ converges in  
$L_p$, as shown, and it extends
$\op{OP}(a_u)$ defined on $\cal S(\Rn)$ by \eqref{axD-eq}. Indeed, since
$|a_u|\le \nrm{F'}{\infty}$, passage to a subsequence (if necessary) and
majorisation gives a.e., for $v\in\cal S(\Rn)$,
\begin{equation}
  a_u(x,D)v(x)=\lim_{N\to\infty} 
  \int \tfrac{e^{\im x\cdot\xi}}{(2\pi)^n}
  (\sum_{j=0}^N m_j\Phi_j)\hat v
  \,d\xi
  = \op{OP}(a_u) v(x).
  \label{auS-eq}
\end{equation}
For $v\in B^{0}_{\infty,1}(\Rn)$, formula \eqref{au-ineq} holds with
$p=\infty$, so by the triangle inequality 
$a_u(x,D)\colon B^0_{\infty,1}\to L_\infty$ is bounded. This gives 
\eqref{Fux-eq} for $u$ in the subspace $B^0_{\infty,1}\subset
L_\infty(\Rn,\R)$. (Restriction to the diagonal will extend
$(u,v)\mapsto a_u(x,D)v$ to all $u$ in $L_\infty(\Rn,\R)$; cf.\ 
\eqref{Fum-eq}.)

\bigskip

The counterexample needed for the theorem is essentially the same
as Ching's construction \cite{Chi72} (that was also analysed in
\cite{Bou88,H88}); with a few convenient modifications this is obtained by
letting 
\begin{equation}
  a(x,\xi)=\sum_{j=1}^\infty 2^{jd}\Phi_j(\xi)e^{-\im x_n2^j}  .
  \label{ching-eq}
\end{equation}
This is in $S^d_{1,1}$, since $2^j\doteq 1+|\xi|$ on the support of
$\Phi_j=\Phi(2^{-j}\cdot)$. 

\begin{lem}
  \label{cex-lem}
For $d\in\R$ there exist symbols $a\in S^{d}_{1,1}(\Rn\times\Rn)$ and
functions 
$\theta_N\in \cal S(\Rn)$ such that, for all $t\in[1,\infty]$, all
$q\in\,]1,\infty]$,
\begin{equation}
  \theta_N\to 0 \text{ in $F^d_{t,q}(\Rn)$ (for $t<\infty$) 
  and in $B^d_{t,q}$}
\end{equation}
while $a(x,D)\theta_N\not\to 0$ in $\cal D'(\Rn)$ for $N\to\infty$.
\end{lem}

\begin{proof}
Take $\theta\in\cal S(\Rn)\setminus\{0\}$ with
$\supp\hat\theta\subset \{\,|\xi|\le\tfrac{1}{20}\,\}$ and let
\begin{equation}
  \hat\theta_N(\xi)=\sum_{j=N}^{N^2} \tfrac{2^{-jd}}{j}
  \hat\theta(\xi-2^je_n)=
  \cal F\bigl(\theta\sum_{j=N}^{N^2} 
  \frac{e^{\im 2^jx_n}}{j2^{jd}} \bigr).
\end{equation}
Since $\Phi_j\equiv1$ on
$\supp \cal F(\theta e^{\im x_n2^j})$, any $q>1$ gives for $N\to\infty$,
\begin{equation}
  \norm{\theta_N}{F^d_{t,q}}=
  \norm{(\sum_{j=N}^{N^2} j^{-q}|\theta|^q)^\fracci1q}{L_t}
  \le \nrm{\theta}{t}(\sum_N^\infty j^{-q})^\fracci1q
  \searrow 0.
\end{equation}
The Besov case is analogous. Because
$\theta_N$ is defined by a finite sum, a direct computation gives for
the above $a(x,\xi)$
\begin{equation}
  a(x,D)\theta_N=(\tfrac{1}{N}+\dots+\tfrac{1}{N^2})\theta(x).
  \label{atN-eq}
\end{equation}
Any $\varphi\in C^\infty_0(\Rn)$ with
$\dual{\theta}{\varphi}=1$ yields
$\dual{a(x,D)\theta_N}{\varphi}\ge \log N$, so clearly
$a(x,D)\theta_N$ does not
tend to zero in the distribution sense.
\end{proof}

Clearly $\hat a(\xi,\eta)=
\sum_{j=1}^\infty (2\pi)^n 2^{jd} \delta_{-2^je_n}(\xi)\Phi_j(\eta)$,
so it is very visible that the condition in \eqref{H-cnd} on the
twisted diagonal is unfulfilled. However, it is also noteworthy 
that $a(x,D)$ moves all frequency contributions
in $\theta_N$ to a neighbourhood of the origin, cf
\eqref{atN-eq}; this is achieved by means of the exponentials 
$e^{\im x\cdot (2^je_n)}$ in $a(x,\xi)$.

\section{On the definition of pseudo-differential operators}
  \label{defn-sect}

Recall first that $a(x,D)u$ is defined for arbitrary symbols $a\in\cal
S'(\R^{2n})$ if one is content with having $u\in\cal S(\Rn)$.
This is via the distribution kernel $K(x,y)=
\cal F^{-1}_{\xi\to z}a(x,\xi)|_{z=x-y}$, 
\begin{equation}
  a(x,D)u(x)= \dual{K(x,\cdot)}{u},\quad\text{for}\quad u\in\cal S(\Rn);
  \label{aK-eq}
\end{equation}
this is just two designations of the functional $\varphi
\mapsto \dual{K_{j,k}}{\varphi\otimes u}$, $\varphi\in \cal S$.
And if, say $a\in S^{\infty}_{1,0}(\Rn\times\Rn)$ suffices,  $a(x,D)u$ is
defined for all $u$ in $\bigcup H^s$ or $\cal S'(\Rn)$; cf~\eqref{aS-eq}. 
The theorem deals with cases between these two extremes, so
it is desirable to explicate how $a(x,D)u$ should be read for 
$a\in S^d_{1,1}(\Rn\times\Rn)$ and $u\in F^d_{p,1}(\Rn)$.

\subsection{Paradifferential techniques}
  \label{para-ssect}
Along with $u_j=\Phi_j(D)u$ it is useful to introduce the auxiliary functions
$\tilde\Phi_j=\Phi_{j-1}+\Phi_j+\Phi_{j+1}$ and set
\begin{equation}
  a_{j,k}(x,\eta)=\cal F^{-1}_{\xi\to x}(\Phi_j\hat a(\cdot,\eta))
  \tilde\Phi_k(\eta) .
\end{equation}
One can then make the ansatz
\begin{equation}
  a(x,D)u(x)=a^{(1)}(x,D)u(x)+a^{(2)}(x,D)u(x)+a^{(3)}(x,D)u(x),
  \label{ansatz-eq}
\end{equation}
when the pair $(a,u)$ is such that the following
series converge in $\cal D'(\Rn)$:
\begin{align}
  a^{(1)}(x,D)u&=\sum_{k=2}^\infty \sum_{j=0}^{k-2} a_{j,k}(x,D)u_k,
  \label{a1xD-eq}  \\
  a^{(2)}(x,D)u&=\sum_{k=0}^\infty \sum_{\substack{j,l=0,1\\ j+l\le1}}
                  a_{k-j,k-l}(x,D)u_{k-l},
  \label{a2xD-eq}  \\
  a^{(3)}(x,D)u&=\sum_{j=2}^\infty \sum_{k=0}^{j-2} a_{j,k}(x,D)u_k.
  \label{a3xD-eq}
\end{align}
The reason for the insertion of $\tilde\Phi_k$ above is that its
compact support yields $a\in S^\infty_{1,1} \implies a_{j,k}\in S^{-\infty}$,
so all terms
$a_{j,k}(x,D)u_k$ make sense for $u\in\cal S'$. 

Clearly $\tilde\Phi_k$ is redundant in $a_{j,k}(x,D)u_k$, for
$\tilde\Phi_k\Phi_k\equiv\Phi_k$ applies in \eqref{axD-eq} for
$u\in \cal S$, so that $\cal S'$-continuity gives
\begin{equation}
 a_{j,k}(x,D)u_k= \op{OP}(\Phi_j(D_x)a(x,\xi)\Phi_k(\xi))u,
  \qquad u\in \cal S'.
  \label{tildephi-eq}
\end{equation}
It is also convenient eg for the later application of
Proposition~\ref{JM-prop} below to have compact support in $\xi$ of $a_{j,k}$.

Note also that when $K_{j,k}$ denotes the kernel of $a_{j,k}(x,D)$, there is a
specific meaning of \eqref{aK-eq}
applied to $a_{j,k}(x,D)$, namely the integral
\begin{equation}
  a_{j,k}(x,D)u_k=\int_{\Rn}K_{j,k}(x,y)u_k(y)\,dy, 
   \quad\text{for}\quad u\in\cal S'(\Rn).
   \label{ajk-eq}
\end{equation}
Indeed, since $\supp a_{j,k}(x,\cdot)\Subset\Rn$,
the Paley--Wiener--Schwartz Theorem and the inequality 
$(1+|y|)^N\le c_N(1+|x|)^N(1+|y-x|)^N$ yield that $K_{j,k}(x,\cdot)$ is 
$\cal O((1+|y|)^N)$ for any $N<0$, while $u_k(y)$ is
so for an $N\ge0$, whence the integral 
exists; \eqref{ajk-eq} follows 
if $u_k\in\cal C^\infty_0(\Rn)$ from Fubini's theorem, so
one can insert $\Psi_m u_k$ in \eqref{ajk-eq} and let $m\to\infty$.

The $a^{(j)}$-series are thus well defined, and they
converge if $a\in S^{\infty}_{1,1}$ and $u\in\cal S$; this
follows from the proof of the theorem in Section~\ref{pf-ssect} below.
Granted this convergence, $a(x,D)u$ defined in \eqref{ansatz-eq} 
is easily seen to equal $\op{OP}(a)u$: indeed, using \eqref{tildephi-eq}
and majorised convergence for $\varphi\in C^\infty_0$,
\begin{equation}
  \begin{split}
  \dual{\sum_{j=1}^3 a^{(j)}(x,D)u}{\varphi}
  &=\lim_{N\to\infty} \dual{\op{OP}(\sum_{j=0}^N \Phi_j(D_x)a)
           (\sum_{k=0}^N u_k)}{\varphi}
\\
  &=\lim_{N\to\infty} \int 
           \tfrac{e^{\im x\cdot\xi}}{(2\pi)^n}
                            \varphi(x)(\Psi_N\hat u)(\xi)
                     \Psi_N(D_x)a(x,\xi) \,d(x,\xi)
\\
  &=\dual{\op{OP}(a)u}{\varphi}.
  \end{split}
  \label{corrsp-eq}
\end{equation}
Therefore any continuity result proved for $a(x,D)$,
with $a\in S^d_{1,1}$, constitutes an extension of 
$\op{OP}(a)\colon \cal S\to\cal S$, in a unique way when $\cal
S$ is dense. This will be the case for the extension to $F^d_{p,1}$
with $1\le p<\infty$ obtained in Section~\ref{pf-ssect} below.
But eg for $u\in B^d_{\infty,1}$ the paradifferential
ansatz above not only describes but also \emph{defines} 
the distribution $a(x,D)u$.

It is important, and essentially known, that the procedure above gives back
the usual pseudo-differential operators, but in lack of a reference a proof
is supplied:
\begin{lem}
  \label{S10-lem}
If $a\in S^\infty_{1,0}(\Rn\times\Rn)$ and $u\in \cal S'(\Rn)$ the series in
\eqref{a1xD-eq}--\eqref{a3xD-eq} converge in $\cal S'(\Rn)$, and
\eqref{ansatz-eq} gives $a(x,D)u$ as defined by \eqref{axD-eq} ff.
\end{lem}

\begin{proof}
With $\scal{\cdot}{\cdot}$ denoting sesqui-linear duality, 
$a(x,D)u$ is the functional $\cal S\ni\varphi\mapsto\scal{u}{b(x,D)\varphi}$ 
for all $u\in \cal S'$ when $a\in S^d_{1,0}$ is given and
$b(x,\xi)=\exp(\im D_x\cdot D_\xi)\overline{a}(x,\xi)$. 
It may be seen as in \cite[Thm.~18.1.7 ff]{H} that 
$a\mapsto b$ is continuous in $S^d_{1,0}$, and this
applies to $a^{(1)}(x,D)$, that by \eqref{tildephi-eq} has symbol
\begin{equation}
  a^{(1)}(x,\xi)=\sum_{k=2}^\infty 
   \Psi_{k-2}(D_x)a(x,\xi)\Phi_k(\xi).
\end{equation}
Indeed, this series converges to $a^{(1)}(x,\xi)$ in the topology of
$S^{d+1}_{1,0}$, so
\begin{equation}
  \begin{split}
  \scal{\sum_{j+2\le k\le N} a_{j,k}(x,D)u}{\varphi}&=
  \sum_{k=2}^N \scal{u}{\bigl(e^{\im D_x\cdot D_\xi}
                              (\Psi_{k-2}(D_x)\overline{a}\Phi_k)\bigr)
                              (x,D)\varphi}
\\
  &\xrightarrow[N\to\infty]{~}
   \scal{u}{\bigl(e^{\im D_x\cdot D_\xi}
                    (\overline{a}^{(1)}(x,\xi))\bigr)(x,D)\varphi}.
  \end{split}
\end{equation}
Here the continuous dependence of the symbol in \eqref{axD-eq}
was also used. Similarly the series for $a^{(2)}$ and $a^{(3)}$ converge,
so the right hand side of \eqref{ansatz-eq} has an action on $\varphi$
equal to 
$\scal{u}{\exp(\im D_x\cdot D_\xi)a(x,D)\varphi}$, 
ie $\scal{a(x,D)u}{\varphi}$.
\end{proof}

\begin{rem}
When \eqref{ansatz-eq} ff is applied to $a_u(x,\xi)$ from
Proposition~\ref{au-prop}, the theorem gives boundedness $a_u(x,D)\colon
F^0_{p,1}\to L_p$, but this equals the operator in Proposition~\ref{au-prop}
in view of \eqref{auS-eq}, \eqref{corrsp-eq} and the density of $\cal S$ in
$F^0_{p,1}$. (The case $B^0_{\infty,1}$ seems to require another treatment.)
\end{rem}

\subsection{Proof of Proposition~\ref{E'-prop}} 
Following the approach above, one can show that for 
$\supp\cal F u\Subset\Rn$,  $a\in S^\infty_{1,1}$,
the series $a^{(1)}(x,D)u$,
$a^{(2)}(x,D)u$ and $a^{(3)}(x,D)u$ all converge 
(the first two are finite sums, and for $a^{(3)}$ one may sum over $j<N$ in
\eqref{ajk-eq} and let $N\to\infty$). 
But there is an equivalent more transparent method,
giving directly that the range is in $O_M(\Rn)$.

If $\hat u\in\cal E'$, $a\in S^\infty_{1,1}$ and $\chi\in C^\infty_0$
equals $1$ in a compact neighbourhood of $\supp\hat u$, then
$b(x,\xi)=a(x,\xi)\chi(\xi)$ is 
in $S^{-\infty}$, hence $f:=b(x,D)u\in O_M$ 
(that $\op{OP}(S^{-\infty})$ maps $\cal S'$ to $O_M$ 
is proved in eg \cite{SRay91}). 
If $\tilde\chi$ is another such cut-off function and $\tilde
b$ the corresponding symbol in $S^{-\infty}$, then $(b-\tilde b)(x,D)v=0$,
for a convolution of $\hat u$ with a sequence of $C^\infty_0$-functions
may produce a sequence $\varphi_k\in \cal S$ that tends to $u$ in $\cal S'$ 
while $(b-\tilde b)(x,D)\varphi_k=0$ eventually.

Moreover, $f=\op{OP}(a)u$ if $u\in\cal F^{-1}C^\infty_0=\cal S\cap\cal
F^{-1}\cal E'$; and 
$f=a(x,D)u$ if $a\in S^\infty_{1,0}$.
Hence $a(x,D)u=f$ is unambiguously defined, and Proposition~\ref{E'-prop} is
proved.

\medskip

Note that, with the set-up of the proof above, it follows from
Lemma~\ref{S10-lem} that the $b^{(m)}$-series converges for $m=1$, $2$, $3$.
But \eqref{tildephi-eq} implies the identity 
\begin{equation}
  b_{j,k}(x,D)u_k=a_{j,k}(x,D)u_k
\end{equation}
for all $j$ and $k$ when $\chi=1$ on a large ball, whence
\begin{equation}
  f=b(x,D)u=a^{(1)}(x,D)u+a^{(2)}(x,D)u+a^{(3)}(x,D)u.
\end{equation}
Thus the given definition is equivalent with the one (mentioned in the
beginning of this section) that consists in
proving directly that \eqref{a1xD-eq}--\eqref{a3xD-eq} all converge for
$\cal Fu\in \cal E'$, hence with the one adopted in
Section~\ref{pf-ssect} below. Consequently any $A\in  \op{OP}(S^d_{1,1})$ is
well defined on the $\cal S'$-subspace
\begin{equation}
  (\sideset{}{'}{\sum_{1\le p<\infty}} F^d_{p,1}(\Rn))+ B^d_{\infty,1}(\Rn)+
  \cal F^{-1}\cal E'(\Rn).
  \label{sumdef-eq}
\end{equation}
($\smash[t]{\sum\nolimits '}$ denotes sums with
only finitely many non-trivial terms.)
Indeed, if $u\in \cal S'$ can be split according to \eqref{sumdef-eq}, the
calculus of limits yields that $a^{(1)}(x,D)u$ etc all converge, with limits
that depend on $u$, but hence not on the splitting. Therefore $a(x,D)u$ is
well defined.

\section{The general borderline case}
  \label{brdl-sect}

\subsection{A pointwise estimate}
To obtain the convergence of the $a^{(j)}(x,D)u$ it is convenient to use
the Hardy-Littlewood maximal function 
\begin{equation}
  M_tf(x)=\sup_{r>0}(\tfrac{1}{|B(x,r)|}\int_{B(x,r)}|f(y)|^t\,dy)^\frac1t,
  \qquad 0<t<\infty.
  \label{Mtf-eq}
\end{equation}
The convolution estimate $|f*g(x)|\le c\nrm{g}{\infty}\cdot
M_1f(x)$, that clearly holds if $f\in L_1^{\op{loc}}$ and 
$g\in L_\infty^{\op{comp}}$, has the following extension to
a `pointwise' estimate for pseudo-differential operators, 
that is central for the present article.

It is remarkable that, in order to get a both weak and flexible requirement
on the symbol, a homogeneous Besov norm of $b(x,\xi)$ is introduced in the
$\xi$-variable, with $x$ considered  as a parameter. Recall here the norm of
the homogeneous Besov space $\dot B^{s}_{p,q}(\Rn)$,
\begin{equation}
  \norm{f}{\dot B^{s}_{p,q}}=
  (\sum_{j=-\infty}^\infty
   2^{sjq}\nrm{\varphi_j(D)f}{p}^q)^{\fracci 1q},
  \label{Bhom-eq}
\end{equation}
where $1=\sum_{j=-\infty}^\infty \varphi_j$ is a partition of unity on
$\Rn\setminus\{0\}$ obtained from $\varphi_j(\xi)=\Phi_1(2^{-j+1}\xi)$.
The $\dot B^s_{p,q}$-norm has the dyadic scaling property:
\begin{equation}
  \norm{f(2^k\cdot)}{\dot B^{s}_{p,q}}=
    2^{k(s-\fracci np)}\norm{f}{\dot B^{s}_{p,q}}
\quad\text{for all}\quad k\in \N.
  \label{dyad-eq}
\end{equation}

\begin{prop}[Marschall's inequality]
  \label{JM-prop}
Let a symbol $b\in\cal S'\cap L_1^{\loc}$ on $\R^{2n}$ and $v\in \cal
S'(\Rn)$ be such that a ball $B(0,2^k)\subset\Rn$, $k\in\Z$, fulfils
\begin{equation}
  \supp\cal Fv \subset B(0,2^k),\qquad
  \supp b\subset \Rn\times B(0,2^k).
  \label{supp-rel}
\end{equation}
The distribution kernel $K$ of $b(x,D)$ is then a locally integrable function, 
ie $K\in L_1^{\loc}(\R^{2n})\cap\cal S'(\R^{2n})$.
Moreover, if for some $t\in\,]0,1]$ 
\begin{equation}
  M_tv(x)\cdot\norm{b(x,2^k\cdot)}{\dot B^{n/t}_{1,t}(\Rn)}
  \in L_1^{\loc}(\Rn),
  \label{Mvb-eq}
\end{equation}
then $K(x,\cdot)v(\cdot)$ is in $L_1(\Rn)$ for a.e.\ $x\in\Rn$, and 
\begin{equation}
  b(x,D)v(x):=\int_{\Rn} K(x,y)v(y)\,dy
  \label{bxDv-eq}
\end{equation}
then defines the action of $b(x,D)$ on $v$ as a function in $L_1^{\loc}(\Rn)$
fulfilling
\begin{equation}
  |b(x,D)v(x)|\le c\norm{b(x,2^k\cdot)}{\dot B^{n/t}_{1,t}(\Rn)}
  M_tv(x)
  \label{JM-ineq}
\end{equation} 
for some constant $c$ independent of $k$.
\end{prop}

\begin{proof}
If $\psi\in C^\infty_0(\R^{2n})$ it follows from the assumption $b\in
L_1^{\loc}$ and the definition of partial Fourier transformation by duality
that
\begin{equation}
  \begin{split}
    \dual{\cal F^{-1}_{\xi\to z}b}{\psi}
  &=
  (2\pi)^{-n}\!\iiint e^{\im z\cdot\xi}b(x,\xi)\psi(x,z)\,d\xi dxdz
\\
  &=\dual{\int \tfrac{e^{\im z\cdot\xi}}{(2\pi)^n}b(x,\xi)\,d\xi}{\psi}.
  \end{split}
\end{equation}
Here the last identity uses that 
$\int \tfrac{e^{\im z\cdot\xi}}{(2\pi)^n}b(x,\xi)\,d\xi$ lies in
$ L_1^{\loc}(\R^{2n})$, hence in $\cal D'$.
Indeed, by Fubini's theorem, 
$\psi(x,z)\int \tfrac{e^{\im z\cdot\xi}}{(2\pi)^n}b(x,\xi)\,d\xi$ is
integrable for any $\psi\in  C^\infty_0$, and in particular if $\psi=1$ on a
large ball. Hence $K(x,y)=\cal F^{-1}_{\xi\to z}b(x,\xi)|_{z=x-y}$
is in $L_1^{\loc}$.

It is now clear that $K(x,y)v(y)$ is measurable, so the following
estimates make sense and, post festum, prove the integrability in view of
\eqref{Mvb-eq}. Indeed, note first that since $\cal F$ 
sends every convolution in $\cal S'*\cal E'$ into a product, cf
\cite[Thm.~7.1.15]{H}, it holds, since $\cal
F^{-1}=(2\pi)^{-n}\bar{\cal F}$, for a.e.\ fixed $x$ that
$(2\pi)^{-n}\cal F^{-1}\bigl((e^{-\im x\cdot \eta}b(x,-\eta))*\hat v\bigr)$
equals $K(x,\cdot)v(\cdot)$, which implies that the latter function 
has spectrum in $B(0,2R)$ for $R=2^k$.
So if $0<t\le 1$, the Nikolski\u\i--Plancherel--Polya inequality, cf
\cite[Thm.~1.4.1]{T2}, gives 
\begin{equation}
  \int |K(x,y)v(y)|\,dy\le c(2R)^{\frac{n}{t}-n}
  (\int |K(x,y)v(y)|^t\,dy)^\frac{1}{t}.
  \label{Knpp-ineq}
\end{equation}
Inserting $1=\sum_{j=-\infty}^\infty \varphi_j(y-x)$ a.e., for
$\varphi_j=\Phi(2^{-j}\cdot)$, 
and using that 
\begin{equation}
  |K(x,y)\varphi_j(y-x)|\le
  \nrm{\cal F^{-1}_{z\to\xi}(\varphi_j(z)\cal F_{\xi\to z}b(x,\xi))}{1}=
  \nrm{\varphi_j(D_\xi)b(x,\cdot)}{1}, 
\end{equation}
the inequality \eqref{Knpp-ineq} and the definition of $M_t v(x)$ give, in
view of \eqref{dyad-eq},
\begin{equation}
\begin{split}
    \int|K(x,y)v(y)|\,dy&\le
    cR^{\frac{n}{t}-n}(\smash[b]{\sum_{j=-\infty}^\infty} 
      \int_{B(x,2^j)}|v(y)|^t\,dy\cdot \tfrac{2^{jn}}{|B(x,2^j)|}
\\
  & \hphantom{\le cR^{\frac{n}{t}-n}(\sum_{j=-\infty}^\infty\int_{B(x,2^j)}|v}
   \times\nrm{K(x,\cdot)\varphi_j(\cdot-x)}{\infty}^t)^{\smash{\frac{1}{t}}}
\\
  &\le cR^{\frac{n}{t}-n}\norm{b(x,\cdot)}{\dot B^{n/t}_{1,t}} M_tv(x)
\\
  &= c \norm{b(x,R\cdot)}{\dot B^{n/t}_{1,t}} M_tv(x).
  \end{split}
  \label{Bdot-ineq}
\end{equation}

If the integrability is exploited to \emph{define} $b(x,D)v$ 
by \eqref{bxDv-eq}, then \eqref{JM-ineq} holds by \eqref{Bdot-ineq} and it
follows from \eqref{Mvb-eq} and the observed measurability
that one gets a distribution in this way.
Note that by \eqref{bxDv-eq} this definition is consistent with the case 
in which $b\in S^\infty_{1,0}$, hence also if $b\in S^\infty_{1,1}$.
\end{proof}

\begin{rem}   \label{Marsch-rem}
Proposition~\ref{JM-prop} requires detailed comments because of overlap with
\cite[Prop.~5(a)]{Mar96}. On the one hand, the estimate \eqref{JM-ineq} is
to my knowledge an original contribution of Marschall; it appeared already
in his thesis \cite[p.~37]{Mar85}, albeit without details.

On the other hand, \cite[Prop.~5(a)]{Mar96} is difficult to follow.
For one thing this is
because of a vague formulation requiring, in addition to \eqref{supp-rel},
$b$ to be ``a symbol $\Rn\times\Rn\to\C$'' 
(replaced by $b\in L_1^{\loc}\cap\cal S'$ in Prop.~\ref{JM-prop}). 
Secondly his proposition is ``singled
out'' after the proof of Proposition~4 there, where the set-up is different
and it furthermore seems to be taken for granted that $b(x,D)v$ \emph{has}
been defined as in \eqref{bxDv-eq} (neither \eqref{Mvb-eq} nor \eqref{bxDv-eq}
was mentioned in \cite{Mar96}); the question of finding conditions assuring
that $b(x,D)v\in \cal D'(\Rn)$ was also not treated, and all in all the
situation is rather more delicate than what \cite{Mar96} gives reason to
believe.   
On these grounds, the details in Proposition~\ref{JM-prop} and its
proof should be well motivated.
\end{rem}

\subsection{Proof of the theorem}
  \label{pf-ssect}
Recall first the Fefferman--Stein inequality 
that the maximal function in \eqref{Mtf-eq}
for $1\le p<\infty$, $1\le q\le\infty$, and any $t\in\,]0,1[$, satisfies
\begin{equation}
  \nrm{(\sum_{k=0}^\infty |M_tf_k|^q)^\fracci1q}{p}\le c
  \nrm{(\sum_{k=0}^\infty |f_k|^q)^\fracci1q}{p}.
  \label{max-ineq}
\end{equation}
For $u\in F^d_{p,1}$ and $f_k:=2^{kd}u_k$, 
the right hand side equals $c\norm{u}{F^d_{p,q}}$.
Taking a fixed $t<1$ such that $\tfrac{n}{t}<n+1$,
this inequality together with Proposition~\ref{JM-prop} will essentially
yield the proof of the theorem. 
 
In addition to \eqref{JM-ineq}, further estimates of $a(x,\xi)$ 
follow from the natural embeddings 
$W^{n+1}_1\hookrightarrow B^{n+1}_{1,\infty}\hookrightarrow 
\dot B^{n/t}_{1,t}$: clearly $2^{kd}\doteq (1+|2^k\eta|)^d$  on
$\supp\tilde\Phi$, since $\tfrac{1}{4}\le|\eta|\le4$,
and since $\Psi_k=\Phi_0+\dots+\Phi_k$,
\begin{equation}
  \begin{split}
  \Norm{\sum_{j=0}^{k-2}a_{j,k}(x,2^k\cdot)}{\dot B^{n/t}_{1,t}}
  &\le 
  \sum_{|\alpha|\le n+1}\Norm{D^\alpha_\xi(\Psi_{k-2}(D_x)a(x,2^{k+2}\cdot)
    \tilde\Phi)}{L_{1,\xi}}
\\
  &\le c \norm{\tilde\Phi}{W^{n+1}_1}
              \nrm{\check\Psi}{1} 2^{kd} \mu_{0,n+1}(a),
  \end{split}
\end{equation}
with $\mu_{l,m}(a)=\sup_{x,\xi;|\beta|\le l, |\alpha|\le m}
 (1+|\xi|)^{-(d-|\alpha|+|\beta|)}|D^\beta_xD^\alpha_\xi a(x,\xi)|$ 
as the seminorms defining the topology on $S^d_{1,1}$. 
Using \eqref{JM-ineq} for each summand in $a^{(1)}(x,D)u$, the
above estimate yields for $k$ in any subset of $\N$,
\begin{equation}
  \begin{split}
    \Nrm{\sum_k\sum_{j=0}^{k-2}a_{j,k}(x,D)u_k}{p}^p
  &\le \int |\sum_k 2^{kd} M_t u_k(x)  |^p\,dx
\\[-3\jot]
  &\qquad
  \times (\sup_{x,k}2^{-kd}
  \Norm{\sum_{j=0}^{k-2}a_{j,k}(x,2^{k+2}\cdot)}{\dot B^{n/t}_{1,t}})^p
\\
  &\le c \mu_{0,n+1}(a)^p \int (\sum_k 2^{kd}|u_k(x)|)^p\,dx.
  \end{split}
  \label{a1-ineq}
\end{equation}
It follows that the series defining $a^{(1)}(x,D)u$
is fundamental in $L_p$ when $u\in F^d_{p,1}(\Rn)$ for $1\le p<\infty$, 
and the same estimate with $k\in \N$ then gives, for $m=1$, 
\begin{equation}
  \nrm{a^{(m)}(x,D)u}{p}\le c \mu_{0,n+1}(a)\norm{u}{F^d_{p,1}}.
  \label{am-ineq}
\end{equation}
The sum $\sum_{j=0}^{k-2}$ may now be replaced by the one
pertinent for $a^{(2)}$, and then essentially the same argument 
yields \eqref{am-ineq} for $m=2$.

To handle $a^{(3)}$, note that 
$0=\int \check\Phi_j(y)y^\alpha\,dy$ for any multiindex $\alpha$ and
$j\ge 1$, so that Taylor's formula for $a(x-y, 2^k\xi)$, with $\xi$
fixed, gives
\begin{multline}
  \Phi_j(D_x)a(x,2^k\xi)=
  \sum_{|\alpha|=N} \tfrac{N}{\alpha!}
 \int (-y)^\alpha\check\Phi_j(y)
\\
  \times
 \int_0^1 (1-\tau)^{N-1}\partial^\alpha_{x} a(x-\tau y, 2^k\xi)\,d\tau dy.
\end{multline}
The factor $(-y)^\alpha$ can absorb a scaling by $2^{jN}$, 
since $\check\Phi_j(y)=2^{jn}\check\Phi(2^j y)$ is a Schwartz function, 
so by the same embeddings as before
\begin{multline}
  2^{jN-kd}\Norm{\Phi_j(D_x)a(x,2^{k+2}\cdot)\tilde\Phi}{\dot B^{n/t}_{1,t}}
\\
 \le c\sum_{\substack{|\alpha|\le n+1\\|\beta|=N}}
    \int |z^\beta\check\Phi(z)|\,dz
    \int (1+|\xi|)^N|D^{\alpha}\tilde\Phi|\,d\xi
     \cdot \mu_{N,n+1}(a).
\end{multline}
This implies that
\begin{equation}
  \begin{split}
  |\sum_{k=0}^{j-2}a_{j,k}(x,D)u_k(x)|&\le
  \sum_{k=0}^{j-2}\Norm{a_{j,k}(x,2^{k+2}\cdot)}{\dot B^{n/t}_{1,t}}
  M_t u_k(x)
\\
  &\le c \mu_{N,n+1}(a)2^{-jN}\sum_{k=0}^{j-2} 2^{kd}M_tu_k(x).
  \end{split}
\end{equation}
Combining this and \eqref{max-ineq} 
with $\sum_{j=0}^\infty 2^{sjq}(\sum_{k=0}^j
|b_j|)^q\le c\sum_{j=0}^\infty 2^{sjq}|b_j|^q$, which for $s<0$,
$0<q\le\infty$, holds for all numbers $b_j$,
cf~\cite[Lem.~3.8]{Y1},
\begin{equation}
  \begin{split}
  \nrm{\sum_{j}\sum_{k=0}^{j-2} a_{j,k}(x,D)u_k}{p}
  &\le c \mu_{N,n+1}(a) 
    \nrm{\sum_{j} 2^{-jN}
    (\sum_{k=0}^j 2^{kd}M_t u_k)}{p}
\\
  &\le c' \nrm{\sum_{j} 2^{(d-N)j} |u_j|}{p}
  \end{split}
  \label{a3-ineq}
\end{equation}
For $N=1$, say, it follows in the same way as above that the series for
$a^{(3)}(x,D)u$ converges in $L_p$; and \eqref{a3-ineq} implies
\eqref{am-ineq} for $m=3$. 
Altogether this yields $\nrm{a(x,D)u}{p}\le c
\mu_{1,n+1}(a)\norm{u}{F^d_{p,1}}$. 

The case $B^d_{\infty,1}(\Rn)$ is analogous, and the necessary
counterexamples were given in Lemma~\ref{cex-lem} above, so the proof of 
the theorem is complete.

\begin{rem}
Even for $p=2$, the above proof 
involves Lebesgue norms on $L_t$ with $0<t<1$ via $M_tu$
($t$ has to be less than the sum-exponent in $F^d_{p,1}$).
\end{rem}

\section{The general continuity properties}
  \label{cont-sect}
This section is devoted to the proof of the corollary and to that of
Proposition~\ref{supp-prop}. 
The main thing will be to prove that the spectra of the general terms
in the $a^{(j)}$-series in Section~\ref{brdl-sect} fulfil 
\begin{gather}
  \supp\cal F\bigl(\sum_{j=0}^{k-2} a_{j,k}(x,D)u_k\bigr)
  \subset
  \bigl\{\,\xi\in\Rn \bigm| \tfrac{1}{5}2^k\le 
                          |\xi|\le \tfrac{5}{4}2^k\,\bigr\}
  \tag{S1}\\
  \supp\cal F\bigl(\sum_{\substack{j,l=0,1\\ j+l\le1}} 
                         a_{k-j,k-l}(x,D)u_{k-l}\bigr)
  \subset
  \bigl\{\,\xi\in\Rn \bigm| 
                          |\xi|\le 4\cdot2^k\,\bigr\}
  \tag{S2}\\
  \supp\cal F\bigl(\sum_{k=0}^{j-2} a_{j,k}(x,D)u_k\bigr)
  \subset
  \bigl\{\,\xi\in\Rn \bigm| \tfrac{1}{5}2^j\le 
                          |\xi|\le \tfrac{5}{4}2^j\,\bigr\}  
  \tag{S3}
\end{gather}
In addition it will be seen that if \eqref{H-cnd} holds, ie for some $C\ge1$,
\begin{equation}
  \hat a(\xi,\eta) =0 \quad\text{where}\quad 
  C(|\xi+\eta|+1)\le |\eta|
  \label{H-cnd'}
\end{equation}
then \mbox{(S2)} may be supplemented by the property that
for $k$ large enough, the set on the left hand side of \mbox{(S2)} is
contained in an annulus,
\begin{equation}
    \supp\cal F\bigl(\sum_{\substack{j,l=0,1\\ j+l\le1}} 
                         a_{k-j,k-l}(x,D)u_{k-l}\bigr)
  \subset
  \bigl\{\,\xi \bigm| \tfrac{1}{4C}2^k\le
                          |\xi|\le 4\cdot2^k\,\bigr\}.
  \tag{S2'}
\end{equation}
 
However, granted that Proposition~\ref{supp-prop} holds,
the inclusions \mbox{(S1)--(S3)} are all easy:
if $\eta$ is in $\supp\hat u_{k}$ clearly 
$\tfrac{11}{20}2^{k}\le|\eta|\le \tfrac{13}{10}2^{k}$, and similarly
if $(\xi,\eta)$ is in the support of  
$\hat a_{j,k}=\Phi_{j}(\xi)\hat a(\xi,\eta)\tilde\Phi_{k}(\eta)$;
then Proposition~\ref{supp-prop} gives that
\begin{equation}
  \xi+\eta\in \supp \cal F a_{j,k}(x,D)u_{k}
  \implies
  |\xi+\eta|\le \tfrac{13}{10}(2^j+ 2^k).
\end{equation}
Since $j-k\in\{0,\pm 1\}$ for each $k$ in the $a^{(2)}$-series, 
$|\xi+\eta|\le\tfrac{39}{10}2^k$, so \mbox{(S2)} holds.
\mbox{(S1)} and \mbox{(S3)} are analogous. \mbox{(S2')}
is seen thus: given \eqref{H-cnd'}, the support rule yields for any
$(\xi,\eta)$ in $\supp \hat a_{j,k}$ so that $\xi+\eta$ is in
the support of $\cal F(a_{k-j,k-l}(x,D)u_{k-l})$, 
that $\eta\in \supp\hat u_{k-l}$, hence  
\begin{equation}
  |\xi+\eta|\ge \tfrac{1}{C} |\eta|-1\ge \tfrac{11}{20C}2^{k-l}-1
  \ge (\tfrac{11}{40C}-2^{-k})2^k;
\end{equation}
here the right hand side is larger than $\tfrac{1}{4C}2^k$ 
for $k>3+\log_2(5C)$. Combined with \mbox{(S2)} this shows \mbox{(S2')}.

\subsection{Proof of the support rule}
  \label{supp-ssect}
Note first that for $f(t)$ in the subspace of $\cal D'$-valued functions 
\begin{equation}
  f\in C(\Rn,\cal D'(\Rn))\subset\cal D'(\R^{2n})
\end{equation}
there is a natural trace at $t=0$ given by $f(0)\in \cal D'(\Rn)$. Moreover,
such $f$ act on $\varphi\in C^\infty_0(\R^{2n})$ by integration, ie if points
in $\R^{2n}$ are written $(t,x)$ for $t$, $x\in \Rn$,
\begin{equation}
  \dual{f}{\varphi}=\int_{\Rn} \dual{f(t)}{\varphi(t,\cdot)}_{\Rn}\,dt.
  \label{cntD'-eq}
\end{equation}
This elementary fact may be seen as in \cite[Prop.~3.5]{JJ00bsp}.

Since $K_2:=\supp\hat v\Subset\Rn$ while $K_1:=\supp \hat b$ is closed in 
$\R^{2n}$,
the set $K$ on the right hand side of \eqref{supp-eq} is closed.
Suppose first that $b$ is a Schwartz function and $\hat v\in C^\infty_0$. 
Then \eqref{Fbv-eq} holds.
To avoid a cumbersome regularisation of $b$ with control of the spectra of
the approximands, note that \eqref{Fbv-eq} also applies to
$b_\tau(x,\eta):=b(x,\eta-\tau)$ for $\tau\in\Rn$. So with the
partially reflected function $R\hat b_\tau(\xi,\eta)=\hat b(\xi,\tau-\eta)$,
and $M=\left(\begin{smallmatrix}
I&0\\-I&I\end{smallmatrix}\right)$, 
\begin{equation}
  \cal F(b_\tau(x,D)v)=\int R\hat b(\xi-\eta,\tau-\eta)\hat v(\eta)\,d\eta
  =R\hat b*((\hat v\otimes\delta_0)\circ M)(\xi,\tau)
  \label{Fbtv-eq}
\end{equation}
Since $\cal S(\R^{2n})\subset S^d_{1,0}(\Rn\times\Rn)$ is dense in the
topology of $S^{d+1}_{1,0}\hookrightarrow \cal S'(\Rn)$, 
and since the right hand side is in the set of convolutions $\cal D'*\cal
E'$ on $\R^{2n}$,
it follows that \eqref{Fbtv-eq}
holds for $b\in S^d_{1,0}$, and then for all $\hat v\in\cal E'(\Rn)$.

As functions of $\tau$, both sides are in
$C(\Rn,\cal D'(\Rn))\subset\cal D'(\R^{2n})$, for it is clear that the
continuity with respect to $\tau$ of the symbol $b_\tau\in S^d_{1,0}$ is
inherited by the left hand side. The right hand side of \eqref{Fbtv-eq} has
support in
\begin{equation}
  \tilde K=\bigl\{\,(\xi,-\theta)+(\eta,\eta) \bigm| (\xi,\theta)\in K_1,
            \eta\in K_2\,\bigr\},
\end{equation}
which is closed when $\hat v\subset\cal E'$; and 
$\tilde K\cap(\Rn\times\{0\})$ equals $K\times\{0\}$.  
So any $\varphi\in C^\infty_0(\Rn)$ with
support in $\Rn\setminus K$ yields a positive distance from
$\tilde K$ to $\supp\varphi\times\{0\}$; 
hence any $\psi\in C^\infty_0(\Rn)$ with $\int\psi=1$
will entail that, eventually, $\varphi(\xi)j^n\psi(j\tau)$ has 
support disjoint from $\tilde K$, hence by \eqref{cntD'-eq} that
\begin{equation}
  \begin{split}
  0&=\int \dual{R\hat b*((\hat v\otimes\delta_0)
     \circ M)(\cdot,\tau)}{\varphi}j^n\psi(j\tau)\,d\tau
\\
  &=\int \dual{\cal F(b_\tau(x,D)v)}{\varphi}j^n\psi(j\tau)\,d\tau
\\
  &\smash[b]{\xrightarrow[j\to\infty]{~}}
    \dual{\cal F(b(x,D)v)}{\varphi}.
  \end{split}
\end{equation}
Finally, for $b\in S^\infty_{1,1}$ it suffices to note that $b(x,D)$,
according to the proof of Proposition~\ref{E'-prop}, 
acts on $v$ as some operator with symbol in $S^{-\infty}$
for which the set $K$ is the same as for $b$.
This completes the proof of Proposition~\ref{supp-prop}.

\subsection{Proof of Corollary~\ref{cont-cor}}

Let for simplicity $u\in F^{s+d}_{p,q}$ and $q<\infty$. 
Mimicking \eqref{a1-ineq} one finds 
\begin{equation}
  \nrm{(\sum_{k=2}^\infty
  2^{skq}|\sum_{j=0}^{k-2}a_{j,k}(x,D)u_k|^q)^{\fracci1q}}{p}
  \le c \mu_{0,n+1}(a)\norm{u}{F^{s+d}_{p,q}}.  
\end{equation}
The conjunction of this estimate and the spectral property \mbox{(S1)}
implies, by the first part of Lemma~\ref{F-lem}, that for $m=1$
\begin{equation}
  \norm{a^{(m)}(x,D)u}{F^{s}_{p,q}}\le c \norm{u}{F^{s+d}_{p,q}}.
  \label{ams-ineq}
\end{equation}
For $m=3$ one can analogously combine \mbox{(S3)} with a
similar modification of \eqref{a3-ineq}, whereby $2^{-jN}$ gets replaced by
$2^{j(s-N)}$, now for $N>s$. 

Concerning $a^{(2)}$, it is easy to show in analogy with \eqref{a1-ineq}
that for the three possible combinations of $j$, $l$ one has
\begin{equation}
  \nrm{(\sum_{k=2}^\infty
  2^{skq}|a_{k-j,k-l}(x,D)u_{k-l}|^q)^{\fracci1q}}{p}
  \le c \mu_{0,n+1}(a)\norm{u}{F^{s+d}_{p,q}}.  
  \label{a2-ineq}
\end{equation}
In view of this, \mbox{(S2)} and the assumption $s>0$,
the criterion for series with spectra in balls (cf Lemma~\ref{F-lem})
gives \eqref{ams-ineq} for $m=2$.

Under the last assumption, \mbox{(S2')} yields for all large $k$ 
that the $k^{\op{th}}$ term in $a^{(2)}$ has spectrum in an annulus, so the
criterion for such series and \eqref{a2-ineq} apply; the remaining
finitely many terms all lie in $\bigcap_{\sigma>0} F^{\sigma}_{p,\infty}$
by the first part of Lemma~\ref{F-lem}. 
Hence \eqref{ams-ineq} holds for all $s\in \R$, all $m$. 

Finally \eqref{Bs-eq} is obtained analogously, and
this completes the proof.

\begin{rem}   \label{Runst-rem}
The mapping properties \eqref{Fs-eq} and
\eqref{Bs-eq} were announced by Runst \cite{Run85ex}, 
albeit with somewhat flawed proofs:
in connection with his Lemma~1 on the basic spectral estimates,
there is \cite[p.~20]{Run85ex} an explicit appeal to a formula like 
\eqref{Fbv-eq} above , but this is not quite enough when symbols in
$S^\infty_{1,1}$ and functions in, say $L_1^{\loc}\setminus\cal S$ are
treated simultaneously.
The same flaw seems to be present in Marschall's work, for
although the spectral properties are claimed in \cite{Mar96} 
without arguments, \eqref{Fbv-eq} was also appealed to in 
\cite[p.~495]{Mar91}. 
However, these shortcomings are only of a technical nature, and they may
easily be remedied by means of the support rule in
Proposition~\ref{supp-prop}, which has sufficiently weak assumptions.
\end{rem}

\section{Final remarks}
  \label{quasi-sect}

The $B^{s}_{p,q}$ and $F^{s}_{p,q}$ spaces have for half a century 
been treated from
many points of view, and it is known that they besides the $H^s_p$
also contain H{\"o}lder--Zygmund classes and other function spaces. The
books of Triebel \cite{T2,T3} account for this and describe the historic
development and priorities. Here the names Besov and Triebel--Lizorkin
spaces are used, since this seems to be common. Also the unifiying
definition by means of Littlewood--Paley decompositions, cf
Section~\ref{Lin-sect}, has been adopted for simplicity.

However, with these definitions, it is well known that the spaces make sense
also for $p$ and $q$ in the interval $\,]0,1[\,$ (although they are then
only quasi-Banach spaces); hence it should be natural to give a brief
treatment of these cases.

\subsection{Exponents $p$ and $q$ in $\,]0,1[\,$}

When extending the continuity results in Corollary~\ref{cont-cor} to 
$p$, $q$ in the full range $\,]0,\infty]$, the
first step could be to replace Lemma~\ref{F-lem} by the well-known criteria
in \cite[2.3.2/2]{RuSi96} or \cite[Thm.~3.7]{Y1}, which both require that 
$s>\max(0,\fracc np-n,\fracc nq-n)$. But there is a somewhat stronger result
showing that the value of $q$ only matters for the ``target space'', while
it is inconsequential for the mere convergence of the series:

\begin{lem}
  \label{Fspq-lem}
Let $s>\max(0,\fracc np-n)$ for $0<p<\infty$ and $0< q\le \infty$ and suppose
$u_j\in \cal S'(\Rn)$ such that, for some $A>0$,
\begin{equation}
  \supp\cal F u_j\subset B(0,A2^j),\qquad
  F(q):=
  \Nrm{(\sum_{j=0}^\infty 2^{sjq}|u_j(\cdot)|^q)^{\fracci1q}}{p}<\infty.
\end{equation}
Then $\sum_{j=0}^\infty u_j$ converges in $\cal S'(\Rn)$ to some $u\in
F^s_{p,r}(\Rn)$ for
\begin{equation}
  r\ge q,\quad r>\tfrac{n}{n+s},
  \label{r-eq}
\end{equation}
and then $\norm{u}{F^s_{p,r}}\le cF(r)$ for some $c>0$ depending on
$n$, $s$, $p$ and $r$.
\end{lem}
\begin{proof}
If $s>\max(0,\fracc np-n,\fracc nq-n)$, then $r=q$ is possible and the claim is
just the usual result. Otherwise $q\le\tfrac{n}{n+s}$, and $F(r)\le
F(q)<\infty$ for $r>q$; if also $r>\tfrac{n}{n+s}$ the standard result gives
the statement.
\end{proof}

Using this, Corollary~\ref{cont-cor} extends to $p$ and $q\in \,]0,\infty]$
as follows:

\begin{cor}
  \label{contpq-cor}
If $a\in S^d_{1,1}(\Rn\times\Rn)$ the corresponding operator $a(x,D)$ is
bounded for $s>\max(0,\fracc np-n)$, $0<p,q\le\infty$ and $r$ as in
\eqref{r-eq}, 
\begin{align}
  a(x,D)&\colon F^{s+d}_{p,q}(\Rn)\to F^s_{p,r}(\Rn)
   \quad (p<\infty);
  \label{Fspr-eq} \\
  a(x,D)&\colon B^{s+d}_{p,q}(\Rn)\to B^s_{p,q}(\Rn).
  \label{Bspq-eq}
\end{align}
If \eqref{H-cnd} holds, then \eqref{Fspr-eq} and \eqref{Bspq-eq} do so for
all $s\in \R$ and $r=q$.
\end{cor}

\begin{proof}
The proof of Corollary~\ref{cont-cor} is easily carried over with the same
relations as in \mbox{(S1)}--\mbox{(S3)}.
The estimates are now made for $t\in \,]0,\min(p,q)[\,$, so that
\eqref{max-ineq} still applies (cf \cite{Y1} eg); but taking $T>0$ such that
$\frac{n}{t}<n+T$ it follows that $W^{n+T}_1$-estimates of the
symbols suffice (they are controlled by the semi-norms $\mu_{0,n+T}(a)$ and
$\mu_{N,n+T}(a)$).

Moreover \eqref{ams-ineq} will need to
have $q$ replaced by $r$ on the left hand side when Lemma~\ref{Fspq-lem} is
invoked instead of Lemma~\ref{F-lem}. And when \eqref{H-cnd} holds,
\mbox{(S2')} still applies, with spectra in annuli except for a finite part
of $a^{(2)}(x,D)u$.
\end{proof}

\begin{rem}
  \label{Torr-rem}
On the one hand, Corollary~\ref{contpq-cor} improves \cite{Tor90} since the
assumption $s>\max(0,\fracc np-n)$ is weaker than his
$s>\max(0,\fracc np-n,\fracc nq-n)$ (the latter is a well-known requirement in
connection with Triebel--Lizorkin spaces, eg it occurs in many places in
\cite{T2, Y1, FJ2, Mar96, RuSi96}, but it is avoided by the sharper
statements in Lemma~\ref{Fspq-lem}).
And the condition \eqref{H-cnd} on the twisted diagonal has not
been extended to the full range of $F^{s}_{p,q}$- and $B^{s}_{p,q}$-spaces
before. On the other  hand, \cite{H88,H89,Tor90} also treat
the continuity from a specific space $F^{s}_{p,q}$ with sufficient
conditions of various kinds, even with some necessary conditions in
\cite{H88,H89}, cf also \cite{H97}; 
the reader is referred to these works for details. 
\end{rem}

\subsection{Acknowledgement}
My thanks are due both to prof.\ G.~Grubb and to prof.\ V.~Burenkov for their
interest in this work.

\providecommand{\bysame}{\leavevmode\hbox to3em{\hrulefill}\thinspace}


\begin{thebibliography}{20}


\bibitem[1]{Bou88}
G. Bourdaud, \emph{Une alg\`ebre maximale d'op\'erateurs
  pseudo-diff\'erentiels}, Comm. Partial Differential Equations \textbf{13}
  (1988), no.~9, 1059--1083.

\bibitem[2]{Chi72}
Chin~Hung Ching, \emph{Pseudo-differential operators with nonregular symbols},
  J. Differential Equations \textbf{11} (1972), 436--447.

\bibitem[3]{CoMe78}
R. Coifman and Y. Meyer, \emph{Au del\`a des op\'erateurs
  pseudo-diff\'erentiels}, Ast\'erisque, vol.~57, Soci\'et\'e Math\'ematique de
  France, Paris, 1978.

\bibitem[4]{FJ2}
M.~Frazier and B.~Jawerth, \emph{{A discrete transform and decomposition of
  distribution spaces}}, J. Func. Anal. \textbf{93} (1990), 34--170.

\bibitem[5]{H}
L.~H{\"o}rmander, \emph{{The analysis of linear partial differential
  operators}}, Grundlehren der mathematischen Wissenschaften, vol. 256, 257,
  274, 275, Springer Verlag, Berlin, 1983, 1985.

\bibitem[6]{H88}
\bysame, \emph{Pseudo-differential operators of type {$1,1$}}, Comm.
  Partial Differential Equations \textbf{13} (1988), no.~9, 1085--1111.

\bibitem[7]{H89}
\bysame, \emph{Continuity of pseudo-differential operators of type {$1,1$}},
  Comm. Partial Differential Equations \textbf{14} (1989), no.~2, 231--243.

\bibitem[8]{H97}
\bysame, \emph{{Lectures on nonlinear differential equations}},
  Math\'ematiques \&\ applications, vol.~26, Springer Verlag, Berlin, 1997.

\bibitem[9]{JJ00bsp}
J.~Johnsen, \emph{Traces of {B}esov spaces revisited}, Z. Anal. Anwendungen
  \textbf{19} (2000), no.~3, 763--779.

\bibitem[10]{JJ04Dcr}
\bysame, \emph{{Domains of type $1,1$ operators: a case for Triebel--Lizorkin
  spaces}}, {C. R. Acad. Sci. Paris S\'er. I Math.} \textbf{339} (2004), no.~2,
  115--118.

\bibitem[11]{Mar85}
J. Marschall, \emph{Pseudo-differential operators with non-regular
  symbols}, Ph.D. thesis, Free University of Berlin, 1985.

\bibitem[12]{Mar91}
\bysame, \emph{Weighted parabolic {T}riebel spaces of product
  type. {F}ourier multipliers and pseudo-differential operators}, Forum Math.
  \textbf{3} (1991), no.~5, 479--511.

\bibitem[13]{Mar96}
\bysame, \emph{Nonregular pseudo-differential operators}, Z. Anal.
  Anwendungen \textbf{15} (1996), no.~1, 109--148.

\bibitem[14]{Mey80}
Y. Meyer, \emph{R\'egularit\'e des solutions des \'equations aux d\'eriv\'ees
  partielles non lin\'eaires (d'apr\`es {J}.-{M}. {B}ony)}, Bourbaki Seminar,
  Vol. 1979/80, Lecture Notes in Math., vol. 842, Springer, Berlin, 1981,
  pp.~293--302.

\bibitem[15]{Mey81}
\bysame, \emph{Remarques sur un th\'eor\`eme de {J}.-{M}. {B}ony}, Proceedings
  of the Seminar on Harmonic Analysis (Pisa, 1980), 1981,
  pp.~1--20.

\bibitem[16]{Bui}
Bui~Huy Qui, \emph{{On Besov, Hardy and Triebel spaces for $0<p\le1$}}, Ark.
  Mat. \textbf{21} (1983), 169--184.

\bibitem[17]{RuSi96}
T.~Runst and W.~Sickel, \emph{{Sobolev spaces of fractional order,
  Nemytski\u\i\ operators and non-linear partial differential equations}},
  Nonlinear analysis and applications, vol.~3, de Gruyter, Berlin, 1996.

\bibitem[18]{Run85ex}
T. Runst, \emph{Pseudodifferential operators of the ``exotic'' class {$L\sp
  0\sb {1,1}$} in spaces of {B}esov and {T}riebel-{L}izorkin type}, Ann. Global
  Anal. Geom. \textbf{3} (1985), no.~1, 13--28.

\bibitem[19]{SRay91}
X. Saint~Raymond, \emph{Elementary introduction to the theory of
  pseudodifferential operators}, Studies in Advanced Mathematics, CRC Press,
  Boca Raton, FL, 1991.

\bibitem[20]{Tor90}
R.~H. Torres, \emph{Continuity properties of pseudodifferential operators
  of type {$1,1$}}, Comm. Partial Differential Equations \textbf{15} (1990),
  1313--1328.

\bibitem[21]{T2}
H.~Triebel, \emph{{Theory of function spaces}}, Monographs in mathematics,
  vol.~78, Birkh{\"a}user Verlag, Basel, 1983.

\bibitem[22]{T3}
\bysame, \emph{{Theory of function spaces II}}, Monographs in mathematics,
  vol.~84, Birkh{\"a}user Verlag, Basel, 1992.

\bibitem[23]{Y1}
M.~Yamazaki, \emph{{A quasi-homogeneous version of paradifferential operators,
  I. Boundedness on spaces of Besov type}}, J. Fac. Sci. Univ. Tokyo Sect. IA,
  Math. \textbf{33} (1986), 131--174.

\end{thebibliography}
\end{document}